\documentclass[12pt]{amsart}
\input amssym.def
\usepackage{xcolor, amsmath}
\usepackage{tikz}

\usepackage{hyperref}
\usepackage{lipsum}
\usepackage{comment}

\newtheorem{theorem}{Theorem}

\newtheorem{observation}{Observation}

\newtheorem{lemma}{Lemma}

\newtheorem{remark}{Remark}

\newtheorem*{acknowledgement*}{Acknowledgement}

\DeclareMathOperator{\GL}{GL}
\DeclareMathOperator{\Gal}{Gal}
\DeclareMathOperator{\Aut}{Aut}

\begin{document}

\title[Torsion groups of Mordell curves over number fields of higher degree]{Torsion groups of Mordell curves over number fields of higher degree}

\author{Tomislav Gu{\v z}vi{\' c}}
\address[]{Tomislav Gu{\v z}vi{\' c}, University of Zagreb, Faculty of Science, Department of Mathematics, Zagreb, Croatia}
\email[]{tguzvic@math.hr}

\author{Bidisha Roy}
\address[]{Bidisha Roy, Institute of Mathematics of the Polish Academy of Sciences, Jana i J{\c e}drzeja {\' S}niadeckich 8, Warsaw 00-656, Poland}
\email[]{brroy123456@gmail.com}



\begin{abstract} Mordell curves over a number field $K$ are elliptic curves of the form $ y^2 = x^3 + c$, where $c \in K \setminus \{ 0 \}$. Let $p \geq 5$ be a prime number, $K$ a number field such that $[K:\mathbb{Q}] \in \{ 2p, 3p \}$ and let $E$ be a Mordell curve defined over $K$. We classify all the possible torsion subgroups $E(K)_{\text{tors}}$ for all Mordell curves $E$ defined over $\mathbb{Q}$ when $[K: \mathbb{Q}] \in \{2p, 3p \}$. 
\end{abstract}

\maketitle
\section{Introduction}

Let $K$ be a number field and let $E/K$ be an elliptic curve. The set on all $K$-rational points of the elliptic curve is denoted by $E(K)$. By a celebrated theorem of Mordell and Weil, it is known that $E(K)$ is a finitely generated abelian group. If we invoke the structure theorem of finitely generated abelian groups on $E(K)$, we get $E(K) \cong E(K)_{\text{tors}} \oplus \mathbb{Z}^r$, where $r \geq0$ is an integer, called the rank of the elliptic curve $E$ over $K$. The group $E(K)_{\text{tors}}$ is called the {\textit torsion subgroup} of $E(K)$. The study of the possible torsion subgroups of a given family of elliptic curve is a well researched topic in algebraic number theory. 

\smallskip

 It is well known that the possible torsion subgroups are of the form $C_m \oplus C_n$, where $m$ and $n$ are positive integers such that $m$ divides $n$. It is natural to try to classify the possibilities of all $E(K)_{\text{tors}}$, where $K$ runs through all number fields of fixed degree and $E$ runs through all elliptic curves defined over $K$. The focus of this paper is to study the growth of $E(K)_{\text{tors}}$, when $[K: \mathbb{Q}] \in \{2p, 3p \}$ for prime number $p \geq 5 $ and for some particular infinite family of elliptic curves. 

\smallskip

Before going into more details, we introduce some notations for our convenience and we briefly mention the relevant history. If we fix an integer $d \geq 1$, then by $\Phi(d)$ we will denote the set of all possible torsion subgroups $E(K)_{\text{tors}}$, where $K$ runs through all number field $K$ of degree $d$ and $E$ runs through all elliptic curves defined over $K$. Many number theorists have been studying these sets in last several years. Starting with the famous result of Mazur \cite{mazur77}, we know that  $$ \Phi (1) = \{ C_n : n = 1, \ldots, 10, 12 \} \cup  \{ C_2\oplus C_{2n} : n =1, \ldots, 4 \}.$$ In other words, if we fix the number field as the set of all rational numbers, then there are only $15$ possibilities of torsion subgroups (up-to isomorphism) for any elliptic curve defined over the set of rational numbers. Later, Kamienny \cite{kamienny} and Kenku-Momose \cite{kenku-momose} independently addressed the case $d=2$. More precisely, they showed that $$ \Phi(2) = \{ C_n : n = 1, \ldots, 16, 18 \} \cup \{ C_2 \oplus C_{2n} : n = 1, \ldots, 6\} \cup \{ C_3 \oplus C_{3n}:n=1,2 \} \cup \{ C_2 \oplus C_4\} .$$ Recently, Derickx, Etropolski, Hoeij, Morrow and Zureick-Brown have determined $\Phi(3)$ in \cite{hoeij}. In general, the set $\Phi(d)$, for $d \geq 4$ is not known. 

\smallskip

Since the sets $\Phi(d)$ are not known explicitly, one can think of reducing the family of elliptic curve to a subfamily. In this notion, Najman \cite{najman} considered the set $\Phi_{\mathbb{Q}}(d) \subseteq \Phi(d)$ which is the set of all possible torsion subgroups of $E(K)_{\text{tors}}$, where $K$ runs through all number fields of degree $d$ and $E$ runs through all elliptic curves defined over $\mathbb{Q}$. For this subfamily, he completely classified the sets $\Phi_{\mathbb{Q}}(d)$, for $d=2, 3$. Later, the sets $\Phi_{\mathbb{Q}}(4)$ and $\Phi_{\mathbb{Q}}(p)$, for $p \geq 5$ is prime, have been determined in \cite{enrique, enrique-najman, najman}. Moreover, in \cite{enrique-najman} it has been shown that  $\Phi_{\mathbb{Q}}(7) = \Phi(1)$ and $\Phi_{\mathbb{Q}}(d) = \Phi(1)$ for any integer $d$ not divisible by $2,3,5$ and $7$. For the sextic number fields, H.B. Daniels and Gonz\'{a}lez-Jim\'{e}nez \cite{jd} and T. Gužvić \cite{tom} have given a partial answer to the classification of $\Phi_{\mathbb{Q}}(6)$.

\smallskip

Apart from the aforementioned family, one can also study a similar thing for another family of elliptic curves, namely the family of elliptic curves with complex multiplication (CM).  Moreover, torsion groups of CM elliptic curves have been studied by many mathematicians in the past several years (see for instance \cite{bc},\cite{bc1},\cite{lr}). In the case of $\mathrm{CM}$-elliptic curves, we denote by ${\Phi}^\text{CM}(d)$ and ${\Phi}^\text{CM}_{\mathbb{Q}}(d)$ the analogue of the sets $\Phi(d)$ and ${\Phi}_{\mathbb{Q}}(d)$ respectively after restricting to CM-elliptic curves. In $1974.$, Olson \cite{ols} completely determined the set ${\Phi}^\text{CM}(1)$. The sets ${\Phi}^\text{CM}(2)$ and ${\Phi}^\text{CM}(3)$ were determined as a special case in a paper by M{\" u}ller et al.  \cite{muller} and by Zimmer et al.  \cite{fswz, pwz} respectively. Recently, Clark et al. \cite{ccrs} have computed the sets ${\Phi}^\text{CM}(d)$ for $4 \leq d \leq 13$. Over odd degree number fields, torsion groups of $\mathrm{CM}$ elliptic curves have been determined by  Bourdon and Pollack in \cite{bourdonpollack}.

 \smallskip
 
 Next, we consider a particular subfamily of CM elliptic curves, namely the set of Mordell curves. The family of all Mordell curves over a number field $K$ consists of elliptic curves that are of the form $y^2 = x^3 + c$,  for some $c \in K$. In the case of Mordell curves, we denote by ${\Phi}^\text{M}(d)$ the set of all possible torsion subgroups of $E(K)_{\text{tors}}$, where $K$ runs through all number fields of degree $d$ and $E$ runs through all Mordell curves defined over $K$. We also define the set ${\Phi}^\text{M}_{\mathbb{Q}}(d)$ to be the intersection ${\Phi}^\text{M}(d) \cap \Phi_{\mathbb{Q}}(d)$ . The study of the sets $\Phi^M (d)$ began long time ago by Knapp  through the determination of the set ${\Phi}^\text{M}(1)$ in \cite{kna}.  Recently, in \cite{dey2}, the set ${\Phi}^\text{M}_{\mathbb{Q}}(d)$ was computed for $d=2$ and for all $d\geq 5$ with $\gcd(d,6) = 1$. 

\smallskip

Recently, in \cite{deyroy} Dey and the second author determined the set ${\Phi}^\text{M}(d)$ and ${\Phi}^\text{M}_{\mathbb{Q}}(d)$ completely for $d=3$ and $6$. Moreover, for a chosen torsion subgroup from the sets ${\Phi}^\text{M}_{\mathbb{Q}}(3), {\Phi}^\text{M}_{\mathbb{Q}}(6)$ and ${\Phi}^\text{M}(3)$, the conditions of arising that subgroup have been completely determined.  The determination of the set  ${\Phi}^\text{M}(6)$ was earlier established as a particular case in \cite{ccrs}.


\smallskip

Motivated by the above, in this paper we study the possible group structures of $E(K)_{\text{tors}}$, where  $[K: \mathbb{Q}]= 2p$ with prime $p \geq 5$  and $E$ a Mordell curve defined over $\mathbb{Q}$. More precisely, we have determined the sets $\Phi^M_\mathbb{Q} (2p)$ and $\Phi^M(3p)$. The techniques we have used here are completely different from the techniques used in \cite{deyroy}. All the computations in this paper were done in Magma \cite{Magma}.

\begin{remark}
Every elliptic curve $E/K$
with $j(E) = 0$ can be written as a Mordell curve and vice versa. Therefore, the classification of torsion groups of Mordell curves is actually the classification of torsion groups of elliptic curves with $j$-invariant equal to $0$.
\end{remark}

  \section{Main Results}
  First we consider the family of Mordell curves of the form $y^2 = x^3 + c$  with $c\in \mathbb Q$. Next we observe that it is enough to assume that $c$ is an integer. For this family of elliptic curves, we find all the possibilities for torsion subgroup of $E(K)$ where $K$ is a number field such that $[K:\mathbb{Q}] \in \{2p, 3p\}$.
  
  \smallskip

For an elliptic curve $E: y^2 = x^3 +c$ with $c\in \mathbb Z$,  we write $c = c_1 t^6$ for some sixth power-free integer $c_1$ and for some nonzero integer $t$. Then $(x,y)$ is a point on the elliptic curve $E_1: y^2 = x^3 + c_1$ if and only if $(t^2x, t^3y)$ is a point on $E$. Thus, it is enough to assume that $c$ is a sixth power-free integer. Here we prove the following results.

\smallskip
\begin{theorem}\label{2p}Let $p$ be a prime number such that $p \geq 5$ and $K$ be any number field of degree $2p$. Also let 
$E: y^2 = x^3 + c$  be a Mordell curve, for any $6$th power-free  element $c$ in $\mathbb{Q}$. Then
    $$E(K)_{\text{tors}} \in \Phi_{\mathbb{Q}}^M(2p) = \{
      C_m, m=1,2,3,6 \}
      \cup \{ C_{2} \oplus C_{2m},  m=1,3 \} \cup \{C_3 \oplus C_{3}\}.$$
\end{theorem}

\begin{theorem}\label{3p}
Let $p$ be a prime number such that $p \geq 5$ and $K$ be any number field of degree $3p$. Also let 
$E: y^2 = x^3 + c$  be a Mordell curve, for any $6$th power-free  element $c$ in $\mathbb{Q}$. Then
$$ E(K)_{\text{tors}} \in \Phi_{\mathbb{Q}}^M(3p) = \{ C_m, \; m=1,2,3,6,9. \}.$$

\end{theorem}

\section{Preliminaries}

Let $E$ be an elliptic curve defined over  a number field $K$ and $n$  be a positive integer. Let $\overline{K}$ be a fixed algebraic closure of $K$. The $n$-torsion subgroup of $ E(\overline{K})$ is denoted by $ E[n]$. More precisely, $E[n] = \{ P \in E (\overline{K}) : n P = \mathcal{O} \}$, where the point at infinity, $\mathcal{O}$ is known as the identity of the group $E( \overline{K})$.  We adjoin all the $x$ and $y$ coordinates of the elements in $ E[n]$ to $K$ and obtain the number field $K(E[n])$. This number field is called  {\it field of definition of the $n$-torsion points.}  In other words, $ K(E[n])$  is the smallest field over which the set $ E[n]$ is defined. The absolute Galois group $ \Gal (\overline{K}/K)$ acts on $ E[n]$ by the map $ \sigma (x, y)  \mapsto (\sigma(x), \sigma(y))$, for each $\sigma \in Gal(\overline{K}/K)$.  This induces a {\it mod $n$ Galois representation attached to $E$} as follows $$ \rho_{E, n} : \Gal (\overline{K}/K) \longrightarrow \Aut (E[n]).$$ It is well known that $ E[n]$ is a free $\mathbb{Z}/ n \mathbb{Z}$ module of rank $2$. After fixing a basis $ \{ P, Q \}$ of $ E[n]$, we can identify $ \Aut(E[n])$ with $\GL_{2} (\mathbb{Z}/ n \mathbb{Z})$. It follows that the aforementioned  map can now be seen as follows $$ \rho_{E, n} : \Gal (\overline{K}/K) \longrightarrow \GL_2 (\mathbb{Z}/ n \mathbb{Z}).$$ The image of the above map, $ \rho_{E, n} (\Gal (\overline{K}/K))$ is a subgroup of $\GL_2 (\mathbb{Z}/ n \mathbb{Z})$ and we denote it by $ G_E(n)$. 

\smallskip

\begin{minipage}{0.3\textwidth} 
\begin{center}
\begin{tikzpicture}[node distance = 1.7cm, auto]
      \node (K) {$K$};
      \node (K1) [above of=K] {$K(R)= K (E[n])^{\mathcal{H}_R}$};
      \node (K2) [above of=K1] {$K(E[n])$};
      \node (K3) [above of=K2] {$\overline{K}$};
      \draw[-] (K) to node {} (K1);
      \draw[-] (K1) to node {} (K2);
      \draw[-] (K2) to node {} (K3);
      \path (K1) edge[bend right] node [right] {$\mathcal{H}_R $} (K2);
      \end{tikzpicture} 
      \end{center}
      \end{minipage} 
\begin{minipage}{0.7\textwidth}   We note that $K (E[n]) = \{ x,y : (x, y) \in E[n] \}$ is a Galois extension over $K$ and since $\ker \rho_{E, n} = \Gal (\overline{K}/ K (E[n]))$, we get that $\rho_{E, n} (\Gal (\overline{K}/K))= G_E(n) \cong  \frac{Gal(\overline{K}/K)}{\ker \rho_{E, n}} =\Gal (K (E[n])/K)$. Let $ R = (x(R), y(R))$ be an element in  $E[n]$. We obviously have $ K(R) = K (x (R), y(R)) \subseteq  K (E[n])$. By Galois theory, there exists a subgroup $ \mathcal{H}_R$ of $ \Gal (K (E[n])/ K)$ such that $ K(R) = K (E[n])^{\mathcal{H}_R}$, the fixed field of $\mathcal{H}_R$. Now we take the image of $\mathcal{H}_R$ in  $ \GL_2 (\mathbb{Z}/ n \mathbb{Z})$ and denote it by $ H_R$. Thus, we have $$ [ K(R) : K ] =[ G_E(n) : H_R].$$ Now we state an important observation.
\end{minipage}\hfill

\begin{observation} 
Let $ E/ K$ be an elliptic curve and let $n$ be a positive integer. Let $ R \in E[n]$ be a point of order $n$. Then $ [K(R): K]$ divides $ |G_E(n)|$. 
\end{observation}

From the above discussion, it is clear that for a given  conjugacy class of $G_E(n)$ in $\GL_2 (\mathbb{Z}/ n \mathbb{Z})$, we can deduce some relevant arithmetic properties of $K(E[n])$. Since $E[n]$ is a free $ \mathbb{Z}/ n \mathbb{Z}$ module of rank $2$, after fixing a basis $\{P, Q \}$ for $E[n]$, we can identify the $n$-torsion points with $ (a, b) \in ( \mathbb{Z}/ n \mathbb{Z}, \mathbb{Z}/ n \mathbb{Z})$. For any $R \in E[n]$, there exist $ a , b \in \mathbb{Z}/ n \mathbb{Z}$ such that $ R = a P + b Q$. Therefore, it is easy to check that $ H_R$ is the stabilizer of $(a, b)$ by the action of $ G_E(n)$ on $( \mathbb{Z}/ n\mathbb{Z})^2$. 

\smallskip

Let $p$ be an odd prime and $ \epsilon = -1$, for $p \equiv 3 \pmod 4$ otherwise let $ \epsilon \geq 2$ be the smallest integer such that $ \left( \frac{\epsilon}{p} \right) = -1$.  We recall the well-known subgroups of $\GL_2 (\mathbb{Z}/ p \mathbb{Z})$ as follows. 

\[ D(a, b) = \begin{bmatrix}
a & 0 \\
0 & b
\end{bmatrix}, \;  M_{\epsilon} (a, b) = \begin{bmatrix}
 a & b \epsilon \\
 b & a
\end{bmatrix},\;  T = \begin{bmatrix}
0 & 1 \\
1 & 0
\end{bmatrix}, \; J = \begin{bmatrix}
1 & 0 \\
0 & -1
\end{bmatrix}, \;
B = \begin{bmatrix}
1 & 1 \\
0 & 1
\end{bmatrix}, \]
 $\text{where} \; a , b \in \mathbb{F}_p$ .

\smallskip

Using the above matrices, we define the following subgroups of $\GL_2 (\mathbb{F}_p)$.
\[ B(p)=\{ D(a,1),D(1,a), B : a \in \mathbb{F}_p ^ \times \},  \]
\[ C_s(p) = \{ D(a, b) : a , b \in \mathbb{F}_p ^\times \},\] \[C_s^+(p) = \{ D(a, b), T \cdot D(a, b) : a, b \in \mathbb{F}_p ^ \times \},\] \[C_{ns} (p) = \{ M _{\epsilon} (a, b) : (a, b) \in \mathbb{F}_p^2, (a, b) \neq (0, 0) \} \] and \[C_{ns}^+ (p) = \{ M_\epsilon (a, b), J \cdot M_{\epsilon}(a, b) : (a, b)  \neq (0, 0) \}. \]

\smallskip

 If $ E/ \mathbb{Q}$ is a CM elliptic cure and $p$ is a prime, then the theory of complex multiplication gives us a lot of information about $ G_E(p) \subseteq \GL_2 (\mathbb{F}_p)$. In the CM case, the possibilities for $ G_E(p)$ are completely understood. We list all the possibilities for $ G_E(p)$, where $E/\mathbb{Q}$ is an elliptic curve with $j(E)=0$. The following theorem is obtained by combining \cite[Theorem 3.6.]{enrique-najman} and \cite[Proposition 1.15.]{zywina}. It is one of the main ingredients for proving Theorem \ref{2p} and Theorem \ref{3p}.

\begin{theorem}\label{enrique and najman}
Let $ E/ \mathbb{Q}$ be a CM elliptic curve and $p$ be a prime. The ring of endomorphisms of  $ E_{ \overline{Q}}$ is an order of conductor $f$ in the ring of integers of an imaginary quadratic field of discriminant $-D$. 
\begin{enumerate}
\item[(i)] If $p=2$, then $ G_E(2)=\GL_{2}(\mathbb{F}_{2})$ or is conjugate in $ \GL_2 (\mathbb{F}_2)$ to $B(2)$. 

\item[(ii)] If $ p >2$ and $(D, f) = (3, 1)$
\begin{enumerate}
\item[(a)] If $ p \equiv 1 \pmod {9}$, then $G_E(p)$ is conjugate in $ \GL_2(\mathbb{F}_p)$ to $ C_s^+(p)$.
\item[(b)] If $ p \equiv 8 \pmod{9}$, then $G_E(p)$ is conjugate in $ \GL_2(\mathbb{F}_p)$ to $ C_{ns}^+(p)$.
\item[(c)] If $ p \equiv 4$ or $7 \pmod 9$, then $ G_E(p)$  is conjugate in $ \GL_2 (\mathbb{F}_p)$ to $ C_s^+(p)$ or to the subgroup $ G^3(p) = \{ D (a, a b^3), T \cdot D(a, ab^3) : a, b \in \mathbb{F}_p^ \times \} \subseteq C_s^+ (p). $
\item[(d)] If $ p \equiv 2 $ or $ 5 \pmod {9}$, then $ G_E(p)$ is conjugate in $ \GL_2(\mathbb{F}_p)$ to $ C_{ns}^+(p)$ or to the subgroup $ G_0(p)$ of $ C_{ns}^+(p)$. 
\item[(e)] If $ p = 3$, then $ G_E(p)$ is conjugate in $ \GL_2 ( \mathbb{F}_3)$ to $3Cs.1.1 $, $C_{s}(3)$, $3B.1.1$, $3B.1.2$ or $B(3)$, where
\[3Cs.1.1=\Big\langle \begin{bmatrix}
1 & 0 \\
0 & 2
\end{bmatrix} \Big\rangle, \; 3B.1.1=\Big\langle \begin{bmatrix}
1 & 0 \\
0 & 2
\end{bmatrix}, \begin{bmatrix}
1 & 1 \\
0 & 1
\end{bmatrix} \Big\rangle,\]
\[ 3B.1.2=\Big\langle \begin{bmatrix}
2 & 0 \\
0 & 1
\end{bmatrix}, \begin{bmatrix}
1 & 1 \\
0 & 1
\end{bmatrix} \Big\rangle.\]
\end{enumerate}
\end{enumerate}
All the cases occur. 
\end{theorem}

The following lemmas (mentioned in \cite{enrique-najman}) give us useful information regarding the arithmetic structure of $G_E(p)$ which will be needed to prove our results.

\begin{lemma}\label{C_ns^+}
Let $E/K$ be an elliptic curve over a number field and $p$ a prime such that $ G_E(p) \cong C_{ns}^+(p)$. Then for a point $P \in E(\overline{K})$ of order $p$, we have $ [\mathbb{Q}(P) : \mathbb{Q}]=p^2-1.$
\end{lemma}

\begin{lemma}\label{two option}
Let $ E/K$ be an elliptic curve over a number field and $p$ a prime such that $ G_E(p)$ is a conjugate to $ C_s^+(p)$ in $ \GL_2(\mathbb{F}_p).$ Then for a point $P \in E(\overline{K})$ of order $p$, we have  $[ \mathbb{Q}(P): \mathbb{Q}] \in \{(p-1)^2, 2(p-1) \}$; both cases can occur.
\end{lemma}

\begin{lemma}\cite[Theorem 5.6.]{enrique-najman}\label{G^3p, G_0p}
Let $E$ be an elliptic curve defined over $\mathbb{Q}$ and let $p$ be a prime number and let $P \in E(\overline{K})$ be a point of order $p$.
If $p \equiv 1 \pmod 3$ and $ G_E(p) \cong G^3(p)$, then \[ [ \mathbb{Q}(P) : \mathbb{Q} ] \in \Big\{  2(p-1), \frac{(p-1)^2}{2},  \frac{2(p-1)^2}{3} \Big\}.\] If $p \equiv 2 \pmod 3$ and $ G_E(p) \cong G_0(p)$ then \[ [ \mathbb{Q}(P) : \mathbb{Q} ] \in \Big\{ \frac{(p-1)^2}{2},  \frac{2(p-1)^2}{3} \Big\}.\]
\end{lemma}

The classification of  ${\Phi}^\text{M}_{\mathbb{Q}}(3)$ is one of many parts handled by Gonz\'{a}lez-Jim\'{e}nez \cite{enrique1} very recently during classification for ${\Phi}^\text{CM}_{\mathbb{Q}}(3)$ for all $13$ imaginary quadratic discriminants of class number $1$. 

\smallskip

 We record $\Phi^M_\mathbb{Q}(2)$  in the following theorem.

\smallskip

\begin{theorem} \cite{dey2}\label{quadraticQ} Let $E/\mathbb{Q}$ be a Mordell curve and let $K$ be a quadratic number field. Then
$$E(K)_{\text{tors}} \in \Phi_{\mathbb{Q}}^M(2) = \{ C_m, m =1,2,3,6 \} \cup \{ C_2 \oplus C_{2m}, m =1, 3 \} \cup C_3 \oplus C_3.$$



\end{theorem}

Another part of this article is to determine torsion subgroups of Mordell curves over number fields of degree $3p$. In order to determine this, we will need the description of the set $\Phi_{\mathbb{Q}}^{M}(3)$ as mentioned below. 

\begin{theorem}\cite{deyroy}\label{cubicQ}
Let $E/\mathbb{Q}$ be a Mordell curve and let $K$ be a cubic number field. Then 
$$ E(K)_{\text{tors}} \in \Phi_{\mathbb{Q}}^M(3) = \{ C_m, m=1,2,3,6,9 \}.$$
\end{theorem}

As we mentioned earlier,  the sets $\Phi^M_\mathbb{Q}(2)$ and $\Phi^M_\mathbb{Q}(3)$ have been completely classified in \cite{dey2} and  \cite{deyroy}, respectively. In this paper, we merely find the set of all possible torsion groups.

\section{Proof of Theorem \ref{2p}}

Let $ R_\mathbb{Q}(d)$ be the set of all primes $p$ such that there exists a number field $K$ of degree $d$, an elliptic curve $ E/ \mathbb{Q}$ such that there exists a point of order $p$ on  $E(K)_{\text{tors}}$. 

Moreover, we will need the following result.

\begin{lemma} \cite[Lemma 1]{tguzvic1}
For any prime number $ p \geq 7$, we have $ R_\mathbb{Q}(2p) = \{ 2,3,5,7 \}.$ Moreover, we have $R_{\mathbb{Q}}(10)=\{ 2,3,5,7,11\}$.
\end{lemma}

It order to complete the proof of Theorem \ref{2p}, it remains to show that $E(K)_{\text{tors}}$ cannot contain a subgroup isomorphic to one of the following:
\[C_{11}, \; C_{7}, \; C_{5}, \;C_{9},\; C_{4},\; C_{3} \oplus C_{6}. \]

\begin{proof}[Proof of Theorem \ref{2p}:] \textcolor{white}{:} \\
\begin{itemize}
    \item Assume that $P_{11} \in E(K)$ is a point of order $11$. By Theorem \ref{enrique and najman}, we obtain that $ G_E(11)$ is conjugate to $ C_{ns}^{+}(11)$ or to $G_0(11)$.  In any case, by Lemma \ref{C_ns^+} and Lemma \ref{G^3p, G_0p}, we get that $ [ \mathbb{Q} (P_{11}) : \mathbb{Q}]$ is divisible by $4$. This is not possible because $ [ \mathbb{Q} (P_{11}) : \mathbb{Q}]$ divides $[K: \mathbb{Q}]=2p$ which is not divisible by $4$. 
    \item Assume that $P_{7} \in E(K)$ is a point of order $7$. By Theorem \ref{enrique and najman}, we have that $G_{E}(7)$ is conjugate to $C_{s}^{+}(7)$ or to $G^{3}(7)$. By Lemma \ref{G^3p, G_0p}, we have that $[\mathbb{Q}(P_{7}):\mathbb{Q}]$ is divisible by $4$. But $\mathbb{Q}(P_{7}) \subseteq K$ and $[K:\mathbb{Q}]=2p$ is not divisible by $4$, a contradiction.
    \item Let $P_{5} \in E(K)$ be a point of order $5$. From Theorem \ref{enrique and najman} (ii)(d), we get that $G_E(5)$ is either conjugate to $C^+_{ns}(5)$ or to the subgroup $G_0(5)$ of $C^+_{ns}(5)$. In any case, by Lemma \ref{C_ns^+} and Lemma \ref{G^3p, G_0p}, we get that $[\mathbb{Q}(P_{5}):\mathbb{Q}]$ is divisible by $4$ which is not possible because $ [\mathbb{Q}(P):\mathbb{Q}]$ divides $[ K: \mathbb{Q}] =2p$.
    \item Let $P_9 \in E(K)$ be a point of order $9$. Then $3P_9$ is a point of order $3$ which we will denote by $ P_3$. By \cite[Proposition 4.6.]{enrique-najman} it follows that $ [\mathbb{Q}(P_9) :\mathbb{Q}(P_3)]$ divides $9$ or  $6$. Furthermore, $ [\mathbb{Q}(P_9) :\mathbb{Q}(P_3)]$ divides $[K:\mathbb{Q}]=2p$. We conclude that $ [\mathbb{Q}(P_9) :\mathbb{Q}(P_3)] \in \{ 1,2\}.$ By Theorem \ref{enrique and najman} (ii)(e) and \cite[Table 1]{enrique-najman}, it follows that $ [\mathbb{Q}(P_3) : \mathbb{Q}]  \in \{ 1, 2\}$. Finally, we conclude that $[\mathbb{Q}(P_{9}):\mathbb{Q}] \in \{ 1,2, 4\}$. If $[\mathbb{Q}(P_{9}):\mathbb{Q}]=4$, then $4$ would divide $2p$, which is impossible. Hence we finally get that $ [\mathbb{Q}(P_9) : \mathbb{Q}] \leq 2$, which contradicts Theorem \ref{quadraticQ}.
    \item Let $P_{4} \in E(K)$ be a point of order $4$. It follows that $2P_{4}$ is a point of order $2$, which will be denoted by $P_{2}$.
As in the previous case, by \cite[Proposition 4.6.]{enrique-najman}
we have $[\mathbb{Q}(P_4) : \mathbb{Q}(P_2)] \in \{1,2,4\}. $ Additionally, by Theorem \ref{enrique and najman} (i), we have that $G_{E}(2)$ is conjugate to either $\GL_{2}(\mathbb{F}_{2})$ or $B(2)$. By \cite[Table 1]{enrique-najman} and \cite[Lemma 5.1]{enrique-najman} it follows that $[\mathbb{Q}(P_2):\mathbb{Q}] \in \{ 1, 2, 3 \}.$ We conclude that $[\mathbb{Q}(P_4):\mathbb{Q}] \in \{ 1,2,3,4,6,8,12 \}$. Since $[\mathbb{Q}(P_4):\mathbb{Q}]$ divides $[K:\mathbb{Q}]=2p$, we have $[\mathbb{Q}(P_4):\mathbb{Q}] \in \{ 1, 2\}$, which is impossible by Theorem \ref{quadraticQ}.
    \item Assume that $C_{3} \oplus C_{6} \subseteq E(K)$. By Theorem \ref{enrique and najman} (ii)(e) and \cite[Table 1]{enrique-najman} it follows that $[\mathbb{Q}(E[3]):\mathbb{Q}] \in \{2,4,6,12 \}.$ Obviously we have $ \mathbb{Q}(E[3]) \subseteq K$, so $[\mathbb{Q}(E[3]) : \mathbb{Q}]$ divides $ [K : \mathbb{Q}] = 2p$. Therefore, we must have $[\mathbb{Q}(E[3]) : \mathbb{Q}]=2.$ Let $P_{2} \in E(K)$ be a point of order $2$. As in the previous case, we conclude that $[\mathbb{Q}(P_{2}):\mathbb{Q}] \in \{ 1,2 \}$.
Since $K$ can contain at most one quadratic subextension it follows that  $\mathbb{Q}(P_{2}) \subseteq \mathbb{Q}(E[3])$. It follows that $C_{3} \oplus C_{6} \subseteq E(\mathbb{Q}(E[3]))$. This is not possible by Theorem \ref{quadraticQ}.
\end{itemize}

Now we show that each group $G$ contained in $\Phi_{\mathbb{Q}}^{M}(3)$ is also contained in $\Phi_{\mathbb{Q}}^{M}(3p)$ by giving a concrete examples of elliptic curve $E/\mathbb{Q}$ with $j(E)=0$ and a number field $L$ such that $[L:\mathbb{Q}]=2p$ that satisfy $E(L)_{tors} \cong G$.
\begin{itemize}
    \item Let $E:y^2=x^3-108$ be an elliptic curve. We have $E(\mathbb{Q})_{tors} \cong \mathcal{O}$. Let $L$ be a number field such that $[L:\mathbb{Q}]=2p$ and $L \cap \mathbb{Q}(E[2],E[3])=\mathbb{Q}$. It follows that $E(L)$ does not contain a point of order $2$ or a point of order $3$. By what we have previously shown, we get $E(L)_{tors} \cong \mathcal{O}$.
    \item Let $E: y^2 =x^3+ 27$ be an elliptic curve and $K = \mathbb{Q}(\sqrt{2})$ be a number filed. Then by \cite [Theorem 1]{dey2}, we get $ E(K)_{\text{tors}} \cong C_2$. Let $L$ be a number field such that $L$ contains $K$, $[L:\mathbb{Q}]=2p$ and $L \cap \mathbb{Q}(E[3]) = \mathbb{Q}$. We conclude that $E(L)$ does not contain a point of order $3$. In order to prove that $E(L)_{tors} \cong C_{2}$, we need to show that $E(L)_{tors} \not\cong C_{2}\oplus C_{2}$. If $E(L)_{tors} \cong C_{2}\oplus C_{2}$, then the polynomial $x^3+27$ would split into linear factors over $L$ which means that $\zeta_{3} \in L$, a contradiction. Therefore, $E(L)_{tors} \cong C_{2}$.
    \item Let $ E: y^2 = x^3 +16$ be an elliptic curve and $ K = \mathbb{Q}(\sqrt{2})$ be a number filed. Then by \cite[Theorem 1]{dey2}, we get $ E(K)_{\text{tors}} \cong C_3$. Let $L$ be a number field such that $L$ contains $K$, $[L:\mathbb{Q}]=2p$ and $L \cap \mathbb{Q}(E[2]) = \mathbb{Q}$. We conclude that $E(L)$ does not contain a point of order $2$. In order to prove that $E(L)_{tors} \cong C_{3}$, we need to show that $E(L)_{tors} \not\cong C_{3}\oplus C_{3}$. If $E(L)_{tors} \cong C_{3}\oplus C_{3}$, then by the properties of Weil pairing we have $\zeta_{3} \in L$, but this contradicts the construction of $L$. Therefore, $E(L)_{tors} \cong C_{3}$.
    \item Let $ E: y^2 = x^3 +1$ be an elliptic curve, $ K = \mathbb{Q}(\sqrt{2})$  and $L$ be be a number filed that contains $K$ and $[L:\mathbb{Q}]=2p$. By \cite[Theorem 1]{dey2}, we get $ E(K)_{\text{tors}} \cong C_6$. From the aforementioned discussion, we have $ E(L)_{\text{tors}} \subseteq \{ C_6, C_2 \oplus C_6 \}$. If $ E(L)_{\text{tors}} \cong C_ 2 \oplus C_6$, then $X^3 +1$ should split completely in $ L$ which is not possible because $ \sqrt{-3} \not\in L$. Hence $E(L)_{\text{tors}} \cong C_6$.
    \item Let $ E: y^2 = x^3 -1$ be an elliptic curve, $ K = \mathbb{Q}(\zeta_3)$ and $L$ be a number field that contains $K$, satisfies $[L:\mathbb{Q}]=2p$ and $L \cap \mathbb{Q}(E[3]) = \mathbb{Q}$. We immediately get that $E(L)$ does not contain a point of order $3$. We conclude that $E(L)_{\text{tors}} \cong C_2 \oplus C_2$.
    \item The elliptic curves $E_{1}: \; y^2=x^3-27$ and $E_{2}: \; y^2+y=x^3$ satisfy $E_{1}(\mathbb{Q}(\zeta_3))_{tors} \cong C_{2} \oplus C_{6}$ and $E_{2}(\mathbb{Q}(\zeta_3))_{tors} \cong C_{3} \oplus C_{3}$.
\end{itemize}

\end{proof}

\section{Proof of Theorem \ref{3p}}
In order to prove Theorem \ref{3p}, we first need to narrow down the possible set of prime numbers $p$ for which there exists a Mordell curve $E/\mathbb{Q}$ and a number field $L$ of degree $3p$ over $\mathbb{Q}$ such that $E(L)$ contains a point of order $p$. We do this with the help of the following lemma.

\begin{lemma} Let $p \neq 3$ be a prime number. Then $ R_{\mathbb{Q}}(3p) = \{2, 3 , 5, 7 \} \cup \{ 19, 43, 67, 163\}$.
\end{lemma}
\begin{proof}
 We know that for all positive integers $n$ and $d$ we have  $R_\mathbb{Q}(d) \subseteq R_\mathbb{Q}(nd).$ Since $ R_\mathbb{Q}(1) = \{ 2,3,5,7 \}$, we have $ R_\mathbb{Q}(1) \subseteq R_\mathbb{Q}(3p)$, for any $ p \geq 5.$
 
 \smallskip
 
 By \cite[Theorem 5.8.]{enrique-najman}, we can see that the only possible elements of $ R_\mathbb{Q}(3p)$, other than $ R_\mathbb{Q}(1)$ lie in the set $\{19, 43, 67, 163 \}$. This is obtained by using the fact that if $P \in E(\overline{K})$ is a point of order $p$, then $[\mathbb{Q}(P):\mathbb{Q}]$ must divide $[K:\mathbb{Q}]=3p$. Hence the proof follows.
\end{proof}

By $ R^M_\mathbb{Q}(d)$, we denote the set of all primes $p$ such that there exits a number field $K$ of degree $d$, a Mordell curve $E/\mathbb{Q}$ such that there exists a point of order $p$ on $E(K)$. In the following lemma, we determine $R^M_\mathbb{Q}(3p)$, where $p \ge 5$ is a prime number.

\begin{lemma}
Let $p \geq 5$ be a prime. Then $R_\mathbb{Q}^M(3p)= \{2, 3\}.$
\end{lemma}

\begin{proof}
 Let $E/\mathbb{Q}$ be a Mordell curve and let $P_{p} \in E(K)$ be a point of order $p$. 
\begin{itemize}
    \item \textbf{$p \in \{7, 43, 67 \}$:} By Theorem \ref{enrique and najman}, we get that $G_E(p)$ is conjugate in $ \GL_2(\mathbb{F}_p)$ to $C_s^+(p)$ or to the subgroup $ G^3(p) = \{ D (a, a b^3), T \cdot D(a, ab^3) : a, b \in \mathbb{F}_p^ \times \} \subseteq C_s^+ (p). $
    In any case, by Lemma \ref{C_ns^+} and Lemma \ref{G^3p, G_0p} we have that $[\mathbb{Q}(P_{p}):\mathbb{Q}]$ is even. This is impossible because $[\mathbb{Q}(P):\mathbb{Q}]$ must divide $[K:\mathbb{Q}]=3p$.
     \item \textbf{$p=5$:} By Theorem \ref{enrique and najman}(iii)(d), we get that $G_E(5)$ is either conjugate to $C^+_{ns}(5)$ or to the subgroup $G_0(5)$ of $C^+_{ns}(5)$. In any case, by Lemma \ref{C_ns^+} and Lemma \ref{G^3p, G_0p} we have that $[\mathbb{Q}(P_{5}):\mathbb{Q}]$ is even, which is impossible. 
     \item \textbf{$p \in \{ 19, 163 \}$:} By Theorem \ref{enrique and najman} (iii)(a), we get $G_E(p)$ is conjugate to $ C_s^+(p)$. Thus, using Lemma \ref{C_ns^+}, we can see that $[\mathbb{Q}(P_{p}):\mathbb{Q}]$ is always divisible by $4$, which is not possible.
 \end{itemize}
\end{proof}

Assume that $C_{m} \oplus C_{m} \subseteq E(K)_{\text{tors}}$. By the properties of the Weil pairing it follows that $\mathbb{Q}(\zeta_{m}) \subseteq K$, so $\phi(m)=[\mathbb{Q}(\zeta_{m}):\mathbb{Q}]$ divides $[K:\mathbb{Q}]=3p$. It follows that $m \in \{ 1,2 \}.$
\newline
It order to complete the proof of Theorem \ref{3p}, first we will show that $E(K)_{\text{tors}}$ cannot contain a subgroup isomorphic to one of the following:
\[C_{4}, \; C_{18}, \;C_{27}, \; C_{2} \oplus C_{2}. \]

\begin{proof}[Proof of Theorem \ref{3p}:] \textcolor{white}{:} \\

\begin{itemize}
    \item Assume that $P_{4} = (x,y) \in E(K)$ is a point of order $4$. Then $E(K)_{\text{tors}}$ has an element of order $2$, which forces $c$ to be a cube, so $c=a^3$ for some $a \in K$. Now we observe that  $y(2P) = 0 \Longleftrightarrow (x(2P))^3 + a^3 = 0$. By \cite[page-105] {silverman}, we know that $x(2P) = \frac{x(x^3-8c)}{4(x^3+c)}$. Using this, we obtain  $ x^6 + 20a^3x^3 -8a^6 = 0 \Longleftrightarrow x^3 = -10a^3 \pm 6a^3\sqrt{3}$. Since $a \in K$, we see that $\sqrt{3} \in K$, which is a contradiction as $K$ is a number field of odd degree. This concludes that $E(K)_{\text{tors}} \not\cong C_4$.
    
    \smallskip
    
    \item  Assume that $P_{18} \in E(K)$ is a point of order $18$. By \cite[Lemma 4.9.]{deyroy}, we get that $9$ divides $[K: \mathbb{Q}]=3p$ which is not possible because $p \ge 5$.
    
    \smallskip
    
    \item Assume that $P_{27} \in E(K)$ is a point of order $27$. Let $\ell=3$ and $n=3$. By \cite[Theorem 1.2]{bourdonpollack}, we have that $\delta = 2$ and  $\mathbb{Z}/ 27 \mathbb{Z}=\mathbb{Z}/\ell^{n}\mathbb{Z}$ appears as the torsion subgroup of $CM$ elliptic curve over an odd degree number field $K$ such that $[K:\mathbb{Q}]=d$ if and only if $d$ is a multiple of $9$. Therefore we conclude that this case is also impossible.
    
    \item Assume that $C_{2} \oplus C_{2} \subseteq E(K)_{\text{tors}}$. By \cite[Proposition 1.15]{zywina}, we have $G_{E}(2)$ is conjugate to $\GL_{2}(\mathbb{F}_{2})$ or $B(2)$. It follows that $|G_{E}(2)| \in \{ 2,6\}$. But since $|G_{E}(2)|=[\mathbb{Q}(E[2]):\mathbb{Q}]$ divides $ [K:\mathbb{Q}]=3p$, we arrive at the contradiction.
\end{itemize}
Now, for each group $G \in \Phi_{\mathbb{Q}}^{M}(3)$, we need to show that it is contained in $\Phi_{\mathbb{Q}}^{M}(3p)$. We do this in the same manner as we did in the proof of Theorem \ref{2p}.

\begin{itemize}
    \item Consider the elliptic curve $E: \; y^2=x^3-108$. Let $L$ be a number field such that $[L:\mathbb{Q}]=3p$ and $L \cap \mathbb{Q}(E[2],E[3])=\mathbb{Q}$. It follows that $E(L)$ does not contain a point of order $2$ nor a point of order $3$. By what we have previously shown, we get $E(L)_{tors} \cong \mathcal{O}$.
        \item Let $ E: y^2 =x^3 +27 $ be an elliptic curve over $\mathbb{Q}$. We choose a cubic number field $K$ such that  $ 4 \cdot 27$ and $108$ are not cubes in $K$. Then, by \cite[Theorem 2.1]{deyroy}, we know that $ E(K)_{\text{tors}} \cong C_2$. Let $L$ be a number field of degree $3p$ that contains $K$. As we have shown earlier, $ E(L)_{\text{tors}} \subseteq \{ \mathcal{O}, C_2, C_3, C_6, C_9 \}$. Thus we can conclude $ E(L)_{\text{tors}} \subseteq \{ C_2, C_6 \}$. If $ E(L)_{\text{tors}} \cong C_6$, then $E(L)_{\text{tors}}$ has a point $P_{3}=(x,y)$ of order $3$.  By \cite[page-105]{silverman}, the $3$rd-division polynomial of $E$ has two factors, namely $x$ and $x^3+108$. By the the construction of $K$, $x^3 +108$ remains irreducible over $L$. Thus we conclude $ E(L)_{\text{tors}} \cong C_2$.
            \item Consider an elliptic curve $E: y^2 =x^3 +4$. We have $ E(\mathbb{Q})_{\text{tors}} \cong C_3$. Let $ L$ be a number field such that $ [L : \mathbb{Q} ] =3p$ and $L \cap \mathbb{Q}(E[2]) = \mathbb{Q}$. This immediately gives us that $E(L)_{tors} \neq C_{6}$. In order to conclude that $E(L)_{tors} = C_{3}$, it remains to show that $E(L)_{tors} \neq C_{9}$.
Consider the $9$th primitive division polynomial (i.e. a polynomial whose roots are precisely the $x$-coordinates of points of order $9$ on $E$) $f_{E,9}(x)$ of $E$. The irreducible factors of $f_{E,9}(x)$ are of degrees $9$ and $27$. Since $[L : \mathbb{Q}]=3p$, we know that these irreducible factors cannot have a root in $L$. Hence, $E (L)_{\text{tors}} \not\cong C_9$. Thus we conclude that $ E(L)_{\text{tors}} \cong C_3$. 
        
    \item Consider the elliptic curves $E_{1}: \; y^2=x^3+16$, $E_{2}: \; y^2=x^3+1$
and a number field $M:= \mathbb{Q}(x_{0})$, where $x_{0}$ is a root of the irreducible polynomial $x^3 -3x^2 +1 =0$. We have $E_{1}(M)_{tors} \cong C_{9}$ and $E_{2}(M)_{tors} \cong C_{6}$. Let $L$ be a number field containing $M$ such that $[L:\mathbb{Q}]=3p$. We have previously shown that $E_{1}(L)_{tors}$ and $E_{2}(L)_{tors}$ are both contained in the set $\{ \mathcal{O}, C_2, C_3, C_6, C_9 \}$. Therefore, we conclude that $E_{1}(L)_{tors} \cong C_{9}$ and $E_{2}(L)_{tors} \cong C_{6}$.

\end{itemize}
\end{proof}

\begin{acknowledgement*} The first author gratefully acknowledges support
from the QuantiXLie Center of Excellence, a project co-financed by the Croatian Government and European Union through the
European Regional Development Fund - the Competitiveness and Cohesion Operational Programme (Grant KK.01.1.1.01.0004) and
by the Croatian Science Foundation under the project no. IP-2018-01-1313. The second author would  like to sincerely thank the Department of Mathematics, University of Zagreb where the project was initiated. The second author  thanks also IMPAN for providing appropriate support to conclude this project. 
\end{acknowledgement*}

\end{document}